\documentclass[11pt]{article}
\usepackage{amsmath}
\usepackage{amssymb}
\usepackage{amsthm}
\usepackage[usenames]{color}
\usepackage{amscd}
\usepackage{indentfirst}
\usepackage[colorlinks=true,linkcolor=blue,filecolor=red,
citecolor=webgreen]{hyperref}
\definecolor{webgreen}{rgb}{0,.5,0}

\def\pont{$\bullet$ }

\newtheorem{theorem}{Theorem}
\newtheorem{lemma}[theorem]{Lemma}

\newtheorem{prop}[theorem]{Proposition}
\newtheorem{remark}[theorem]{Remark}

\begin{document}

\title{{\bf Some properties of a sequence defined with the aid of prime numbers}}
\author{Br\u{a}du\c{t} Apostol, \fbox{Lauren\c{t}iu Panaitopol}, Lucian Petrescu, and L\'aszl\'o T\'oth}
\date{}
\maketitle

\centerline{Journal of Integer Sequences {\bf 18} (2015), Article 15.5.5}

\begin{abstract} For every integer $n\ge 1$ let $a_n$ be the smallest positive integer such that $n+a_n$ is prime.
We investigate the behavior of the sequence $(a_n)_{n\ge 1}$, and
prove asymptotic results for the sums $\sum_{n\le x} a_n$,
$\sum_{n\le x} 1/a_n$ and $\sum_{n\le x} \log a_n$.
\end{abstract}



\section{Introduction} For every integer $n\ge 1$ let $a_n$
be the smallest positive integer such that $n+a_n$ is prime. Here
$a_1=1$, $a_2=1$, $a_3=2$, $a_4=1$, $a_5=2$, $a_6=1$, $a_7=4$, etc.
This is sequence A013632 in Sloane's Online Encyclopedia of Integer
Sequences \cite{OEIS}. For $n\ge 2$, $a_n$ is the smallest positive
integer such that $\gcd(n!,n+a_n)=1$. In this paper we study the
behavior of the sequence $(a_n)_{n\ge 1}$, and prove asymptotic
results for the sums $\sum_{n\le x} a_n$, $\sum_{n\le x} 1/a_n$ and
$\sum_{n\le x} \log a_n$.

We are going to use the following standard notation:

\pont $\pi(x)$ is the number of primes $\le x$,

\pont $\pi_2(x)$ is the number of twin primes $p,p+2$ such that
$p\le x$,

\pont $p_n$ is the  $n$-th prime,

\pont $d_n=p_{n+1}-p_n$,

\pont $f(x) \ll g(x)$ means that $|f(x)|\le C g(x)$, where $C$ is an absolute constant,

\pont $g(x) \gg f(x)$ means that $f(x) \ll g(x)$,

\pont $f(x) = F(x)+ O(g(x))$ means that $f(x) - F(x) \ll g(x)$,

\pont $f(x) \asymp g(x)$ means that $c f(x) \le g(x) \le Cf(x)$ for some positive absolute constants $c$ and $C$,

\pont $f(x) \sim g(x)$ means that $\lim_{x\to \infty} f(x)/g(x)=1$.

We will apply the following known asymptotic results concerning the
distribution of the primes:
\begin{equation*}
\pi(x) \sim \frac{x}{\log x}, \quad p_n\sim n\log n \quad {\text{\rm (Prime number theorem)}},
\end{equation*}
\begin{equation} \label{H}
\sum_{p_n\le x} d_n^2\ll x^{23/18+\varepsilon} \quad {\text{\rm for
every $\varepsilon >0$ (unconditional result of Heath-Brown
\cite{H})}},
\end{equation}
\begin{equation} \label{S}
\sum_{p_n\le x} d_n^2\ll x (\log x)^3 \quad  {\text{\rm (assuming
the Riemann hypothesis, result of Selberg \cite{S})}},
\end{equation}
\begin{equation} \label{P}
\left(\frac{d_2d_3\cdots d_n}{(\log 2)(\log 3)\cdots (\log
n)}\right)^{1/n}\asymp 1 \quad {\text{\rm (due to Panaitopol
\cite[Prop.\ 3]{P})}}.
\end{equation}

This research was initiated by Lauren\c{t}iu Panaitopol
(1940--2008), former professor at the Faculty of Mathematics,
University of Bucharest, Romania. The present paper is dedicated to his
memory.

\section{Equations and identities}

By the definition of $a_n$, for every $n\ge 1$ we have
$n+a_n=p_{\pi(n)+1}$, that is
\begin{equation} \label{exp_form}
a_n= p_{\pi(n)+1}-n.
\end{equation}

From \eqref{exp_form} we deduce that for every $k\ge 1$,
\begin{equation} \label{values}
a_{p_k}= p_{k+1}-p_k, a_{p_k+1}= p_{k+1}-p_k-1,\ldots, a_{p_{k+1}-1}=1.
\end{equation}

\begin{prop} For every integer $a\ge 1$  the equation $a_n=a$ has infinitely many solutions.
\end{prop}

\begin{proof} Let $A_k=\{1,2,\ldots, p_{k+1}-p_k\}$. Since $\limsup_{k\to \infty}
(p_{k+1}-p_k)=\infty$, it follows from \eqref{values} that for every
integer $a\ge 1$ there exist infinitely many integers $k\ge 1$ such
that $a\in A_k$, whence the equation $a_n=a$ has infinitely many
solutions.
\end{proof}

Now we compute the sum $\displaystyle S_n=\sum_{i=1}^n a_i$.

\begin{prop} For every prime $n\ge 3$ we have
\begin{equation} \label{sum_prime}
S_n= \frac1{2}\left(2p_{\pi(n)+1}-p_{\pi(n)}+\sum_{k=1}^{\pi(n)-1} d_k^2 \right),
\end{equation}
and for every composite number $n\ge 4$,
\begin{equation} \label{sum_composite}
S_n= \frac1{2}\left(p_{\pi(n)}^2+ 2(n+1-p_{\pi(n)}) p_{\pi(n)+1}
+\sum_{k=1}^{\pi(n)-1} d_k^2 -n^2-n \right).
\end{equation}
\end{prop}

\begin{proof} If $n\ge 3$ is a prime, then $n=p_m$ for some $m\ge 2$. By using
\eqref{exp_form},
\begin{align*}
S_n & = \sum_{i=1}^n \left(p_{\pi(i)+1}-i\right) \\
& = 2+3+(5+5)+\cdots+(p_m-p_{m-1})p_m+p_{m+1}-\frac{n(n+1)}{2} \\
& = 2+ \sum_{k=2}^m p_k(p_k-p_{k-1})+p_{m+1}- \frac{n(n+1)}{2}\\
& = \frac1{2}\left(p_1^2+ 2 \sum_{k=2}^m p_k^2 - 2\sum_{k=2}^m p_k
p_{k-1} +2p_{m+1}- n^2-n\right)\\
& = \frac1{2}\left(2p_{m+1}-n+ \sum_{k=1}^{m-1} (p_{k+1}-p_k)^2
\right)
\end{align*}
and \eqref{sum_prime} follows by using that $m=\pi(n)$.

Now let $t\ge 4$ be composite. Let $m\ge 2$ be such that
$p_m<t<p_{m+1}$. By applying \eqref{sum_prime} for $n=p_m$, where
$m=\pi(n)=\pi(t)$, we deduce
\begin{align*}
S_t & = S_n+ \sum_{i=n+1}^t a_i = S_n + \sum_{i=n+1}^t \left(p_{\pi(i)+1}-i \right)\\
& = \frac1{2}\left(2p_{\pi(t)+1}-p_{\pi(t)}+\sum_{k=1}^{\pi(t)-1}
(p_{k+1}-p_k)^2 \right)+ \frac{(2p_{\pi(t)+1}-n-t-1)(t-n)}{2}\\
& = \frac1{2}\left(2p_{\pi(t)+1}-p_{\pi(t)}+\sum_{k=1}^{\pi(t)-1}
(p_{k+1}-p_k)^2 + 2p_{\pi(t)+1}(t-n)-t^2-t+n^2+n\right)\\
& = \frac1{2}\left(p_{\pi(t)}^2+ 2(t+1-p_{\pi(t)}) p_{\pi(t)+1}
+\sum_{k=1}^{\pi(t)-1} (p_{k+1}-p_k)^2-t^2-t \right)
\end{align*}
and \eqref{sum_composite} is proved.
\end{proof}

\begin{remark} If $n$ is prime, then \eqref{sum_composite} reduces to \eqref{sum_prime}. Therefore,
the identity \eqref{sum_composite} holds for every integer $n\ge 3$.
\end{remark}

Next we compute the product $\displaystyle P_n=\prod_{i=1}^n a_i$.

\begin{prop} For every prime $n\ge 3$ we have
\begin{equation} \label{prod_prime}
P_{n-1} = \prod_{k=1}^{\pi(n)-1} d_k!,
\end{equation}
and for every composite number $n\ge 4$,
\begin{equation}
\label{prod_composite} P_{n-1} = \prod_{k=1}^{\pi(n)-1} d_k!
\prod_{k=1}^{n-p_{\pi(n)}} (p_{\pi(n)+1}-p_{\pi(n)}-k+1).
\end{equation}
\end{prop}

\begin{proof} Let $n=p_m\ge 3$ be a prime. By using
\eqref{values},
\begin{align*}
P_{n-1} & = \prod_{i=2}^m (p_i-p_{i-1})!= \prod_{i=1}^{m-1} (p_{i+1}-p_i)!,
\end{align*}
which proves \eqref{prod_prime}.

Now let $t\ge 4$ be composite such that
$p_m<t<p_{m+1}$. By applying \eqref{prod_prime} for $n=p_m$, where
$m=\pi(n)=\pi(t)$, we deduce
\begin{align*}
P_{t-1} & = P_{n-1} \prod_{i=n}^{t-1} a_i = P_{n-1} \prod_{i=n}^{t-1} \left(p_{\pi(i)+1}-i \right)\\
& = \prod_{k=1}^{\pi(t)-1} d_k! \prod_{j=1}^{t-p_m} \left(p_{m+1}-p_m-j+1 \right) \\
& = \prod_{k=1}^{\pi(t)-1} d_k! \prod_{k=1}^{t-p_{\pi(t)}} \left(p_{\pi(t)+1}-p_{\pi(t)}-k+1 \right)
\end{align*}
and \eqref{prod_composite} is proved.
\end{proof}

\begin{remark} If $n$ is prime, then the second product in \eqref{prod_composite} is empty and
\eqref{prod_composite} reduces to
\eqref{prod_prime}. Hence the identity \eqref{prod_composite} holds for every integer $n\ge 3$.
\end{remark}

\section{Asymptotic results}

\begin{theorem} For every $\varepsilon >0$,
\begin{equation} \label{asymp_sum}
x\log x \ll \sum_{n\le x} a_n \ll x^{23/18+\varepsilon},
\end{equation}
where $23/18\doteq 1.277$. If the Riemann hypothesis is true, then the
upper bound in \eqref{asymp_sum} is $x(\log x)^3$.
\end{theorem}

\begin{proof} Let $x\ge 2$ and let $p_k\le x< p_{k+1}$. By using \eqref{sum_prime}
for $n=p_{k+1}$,
\begin{align*}
\sum_{n\le x} a_n & \le \sum_{i=1}^{p_{k+1}} a_i =
\frac1{2}\left(2p_{k+2}-p_{k+1} +\sum_{i=1}^k d_i^2
\right)\\
& \ll p_{k+2} + \sum_{p_i\le x} d_i^2.
\end{align*}

Taking into account the estimate \eqref{H} due to Heath-Brown, and
the fact that $p_{k+2}\sim p_k \le x$ we get the unconditional upper
bound in \eqref{asymp_sum}. If the Riemann hypothesis is true, then
by using Selberg's result \eqref{S} we obtain the upper bound
$x(\log x)^3$.

Now, for the lower bound we use the trivial estimate
\begin{equation*}
\sum_{p_n\le x} d_n^2 \gg x\log x,
\end{equation*}
which follows from the inequality between the arithmetic and
quadratic means. We deduce that
\begin{align*}
\sum_{n\le x} a_n & \ge \sum_{i=1}^{p_k} a_i =
\frac1{2}\left(2p_{k+1}-p_k +\sum_{i=1}^{k-1} d_i^2
\right)\\
& \gg  \sum_{p_i\le p_{k-1}} d_i^2 \gg p_{k-1} \log
p_{k-1} \sim x\log x,
\end{align*}
since $p_{k-1}\sim k\log k$ and $k=\pi(x)\sim x/\log x$, $\log k
\sim \log x$.
\end{proof}

To prove our next result we need the following

\begin{lemma} We have
\begin{equation} \label{asympt_lemma}
\sum_{2\le n\le x} \log d_n = x\log \log x + O(x).
\end{equation}
\end{lemma}

\begin{proof} The inequalities \eqref{P} can be written as
\begin{equation*}
c n < \sum_{i=2}^n \log d_i -\sum_{i=2}^n \log \log i < C n
\end{equation*}
for some positive absolute constants $c$ and $C$. Now
\eqref{asympt_lemma} emerges by applying the well known asymptotic
formula
\begin{equation*}
\sum_{2\le n\le x} \log \log n = x\log \log x + O(x).
\end{equation*}
\end{proof}

\begin{theorem} We have
\begin{equation}\label{asymp_harm}
\sum_{n\le x} \frac1{a_n} = \frac{x\log \log x}{\log x} +
O\left(\frac{x}{\log x}\right).
\end{equation}
\end{theorem}

\begin{proof} For $x=p_m-1$ ($m\ge 2$) we have by \eqref{values},
\begin{equation*}
\sum_{n\le p_m-1} \frac1{a_n} = 1+ \sum_{i=2}^m
\left(1+\frac1{2}+\cdots+\frac1{p_i-p_{i-1}}\right).
\end{equation*}

For an arbitrary $x\ge 3$ let $p_k$ ($k\ge 2$) be the prime such
that $p_k\le x< p_{k+1}$. Using the familiar inequalities
\begin{equation*}
\log m < 1+\frac1{2}+\cdots+\frac1{m}\le 1+\log m \quad (m\ge 1)
\end{equation*}
we deduce
\begin{equation*}
\log (p_i-p_{i-1})< 1+\frac1{2}+\cdots+\frac1{p_i-p_{i-1}} \le
1+\log (p_i-p_{i-1}) \quad (i\ge 2)
\end{equation*}
and
\begin{equation*}
1+\sum_{i=2}^k \log (p_i-p_{i-1}) +\frac1{d_k} < \sum_{n\le p_k-1} \frac1{a_n}+ \frac1{a_{p_k}}
\end{equation*}
\begin{equation*}
\le \sum_{n\le x} \frac1{a_n} \le \sum_{n\le p_{k+1}-1} \frac1{a_n}
\le 1+ k + \sum_{i=2}^{k+1} \log (p_i-p_{i-1}).
\end{equation*}

By \eqref{asympt_lemma} we obtain
\begin{equation*}
\sum_{n\le x} \frac1{a_n} = k\log \log k + O(k),
\end{equation*}

Here $k=\pi(x)\sim x/\log x$, $\log k
\sim \log x$ and we deduce \eqref{asymp_harm}.
\end{proof}

\begin{theorem} One has
\begin{equation*}
x \ll \sum_{n\le x} \log a_n  \ll x\log x.
\end{equation*}
\end{theorem}

\begin{proof} For an arbitrary $x\ge 3$ let $p_k$ ($k\ge 2$) be the prime such
that $p_k\le x< p_{k+1}$. Using the elementary inequalities
\begin{equation*}
m \log m - m + 1 \le \log m! \le m\log m \quad (m\ge 1)
\end{equation*}
we deduce by applying \eqref{prod_prime} that
\begin{equation*}
\sum_{n\le x} \log a_n \le \sum_{n\le p_{k+1}-1} \log a_n =
\sum_{i=1}^k \log d_i! \le \sum_{i=1}^k d_i \log d_i
\end{equation*}
\begin{equation*}
< \sum_{i=1}^k d_i \log p_i < (\log p_k) \sum_{i=1}^k d_i < (\log
p_k)p_{k+1},
\end{equation*}
where we also used that $d_i=p_{i+1}-p_i<p_i$ by Chebyshev's
theorem. Here
\begin{equation} \label{k_x}
p_k\sim k\log k, \quad k=\pi(x)\sim x/\log x, \quad \log k \sim \log
x,
\end{equation}
and we obtain the upper bound $x\log x$.

On the other hand,
\begin{equation*}
\sum_{n\le x} \log a_n > \sum_{n\le p_k-1} \log a_n  =
\sum_{i=1}^{k-1} \log d_i!
\end{equation*}
\begin{equation*}
>  \sum_{i=1}^{k-1} (d_i \log d_i-d_i+1 ) = \sum_{i=2}^{k-1} d_i\log d_i - p_k+
k+1.
\end{equation*}

Here
\begin{align*}
\sum_{i=2}^{k-1} d_i \log d_i & = \sum_{\substack{i=2\\ d_i\ge
3}}^{k-1} d_i\log d_i + 2\log 2 \sum_{\substack{i=2\\ d_i=2}}^{k-1}
1 \\
& \ge (\log 3) \sum_{\substack{i=2\\ d_i\ge 3}}^{k-1} d_i + (2\log
2)\pi_2(k-1)\\
& = (\log 3)\left( \sum_{i=2}^{k-1} d_i - \sum_{\substack{i=2\\
d_i=2}}^{k-1} d_i \right) + (2\log 2)\pi_2(k-1) \\
& = (\log 3)\left(p_k-p_2  - 2\pi_2(k-1) \right) + (2\log
2)\pi_2(k-1)\\
& = (\log 3)p_k -2\log(3/2)\pi_2(k-1)-3\log 3 \\
& > (\log 3)p_k -2\log(3/2)k- 3\log 3,
\end{align*}
where it is sufficient to use the obvious estimate $\pi_2(k-1)< k$.
Note that $\log 3\doteq 1.09$, $2\log(3/2)\doteq 0.81$, $3\log
3\doteq 3.29$.

We deduce that
\begin{equation*}
\sum_{n\le x} \log a_n > 0.09 p_k - 3.
\end{equation*}

Now, \eqref{k_x} gives the lower bound $x$.
\end{proof}

\vskip6mm

\noindent B.~Apostol \\ Pedagogic High School "Spiru Haret", Str. Timotei Cipariu 5, RO--620004 Foc\c{s}ani, Romania
\\ E-mail: \verb"apo_brad@yahoo.com"
\medskip

\noindent L.~Petrescu \\ Technical College "Henri Coand\u{a}", Str. Tineretului 2, RO--820235 Tilcea, Romania
\\ E-mail: \verb"petrescuandreea@yahoo.com"
\medskip

\noindent L. T\'oth \\
Department of Mathematics, University of P\'ecs \\ Ifj\'us\'ag \'utja 6,
H--7624 P\'ecs, Hungary \\ E-mail: \verb"ltoth@gamma.ttk.pte.hu"


\begin{thebibliography}{99}

\bibitem{H} D.~R.~Heath--Brown, The differences between consecutive primes, III,
{\it J. London Math. Soc.} (2) {\bf 20} (1979), 177--178.

\bibitem{P} L.~Panaitopol, Properties of the series of differences of prime numbers,
{\it Publ. Centre de Rech. Math. Pures Neuch\^{a}tel
(Serie 1)} {\bf 31} (2000), 21--28.

\bibitem{S} A.~Selberg, On the normal density of primes in small intervals, and the difference between
consecutive primes, {\it Arch. Math. Naturvid.} {\bf 47} (1943), No. 6, 87--105.

\bibitem{OEIS} N.~J.~A.~Sloane, The On-Line Encyclopedia of Integer Sequences.
\url{http://oeis.org}

\end{thebibliography}
\end{document}